\documentclass[10pt, leqno]{amsart}

\usepackage{Style}

\title{Resonance of rank-two vector bundles over elliptic curves}

\author{C\u alin Spiridon}

\address{University of Bucharest, Faculty of Mathematics and Informatics, Academiei Str. 14 Bucharest, Romania \& "Simion Stoilow" Institute of Mathematics of the Romanian Academy, P.O. Box 1-764, Bucharest, Romania}

\email{cspiridon@imar.ro}

\thanks{The author has been partly funded by the project PNRR-III-C9-2022-I8 "Cohomological Hall algebras of smooth surfaces and applications" - CF 44/14.11.2022.}
\thanks{The author thanks professor Marian Aprodu for helpful discussions.}

\date{\today}

\begin{document}

\begin{abstract}
In this note, we study the resonance variety of rank-two vector bundles over an elliptic curve. Our approach is based on analyzing the flattening stratification of the resonance, as introduced in \cite{aprodu_linear_2025}. We also investigate the linear section of the Grassmann variety $\gr(2,n)$ from which the resonance is constructed through the lens of its corresponding flattening stratification.
\end{abstract} 

\maketitle


\section{Introduction}

Initially arising in the study of hyperplane arrangements, the resonance varieties of a space are cohomological jump loci that reflect an interesting interplay of topology, geometry and combinatorics. 
Inspired by the first resonance variety of a group in the framework of geometric group theory, Papadima and Suciu \cite{papadima_vanishing_2015} introduced, in a purely algebraic setup, the resonance variety of a pair $(V,K)$, where $K$ is a linear subspace in the second exterior power of a vector space $V$. As immediately pointed out by the mentioned authors, these varieties are support loci for some naturally defined graded modules over the symmetric algebra of $V$. We indicate \cite{aprodu_green_2019}, \cite{aprodu_topological_2022}, \cite{aprodu_koszul_2024}, \cite{aprodu_reduced_2024} for a systematic and multi-angled study of these objects, together with some remarkable applications.

A fruitful source of examples comes from algebraic vector bundles. This notion was introduced in \cite{aprodu_koszul_2024} and we briefly recall the definition. Let $X$ be a smooth projective variety and $E$ a vector bundle on $X$. Two linearly independent global sections of $E$ are said to \emph{resonate} if they generate a rank-one subsheaf of $E$. Such global sections are called \emph{resonant}. The locus $\r(E)$ of classes of resonant sections in the projectivization $\p H^0(X,E)$ is called the \emph{resonance variety} associated to the vector bundle $E$. It turns out that, in some cases, these varieties can reflect certain geometric properties of the vector bundle. It is also worth noting that, although not defined as such at the time, the resonance of a vector bundle already appeared in the work of S. Mukai \cite{mukai_curves_1992} or C. Voisin \cite{voisin_greens_2002}, \cite{voisin_greens_2005}.

Building on \cite{aprodu_reduced_2024}, it was shown in \cite{aprodu_linear_2025} that the resonance of any vector bundle admits a natural finite stratification, called the \emph{flattening stratification}, which is equipped with a family of morphisms to a Quot scheme. When $X$ is a projective curve, the situation is even more favorable, since these strata admit algebraic maps to some Brill-Noether loci, that are better understood. In \cite{aprodu_linear_2025}, a detailed study was carried out in the case of rank-two bundles over the projective line. This note aims to investigate the next natural case, that of rank-two bundles over elliptic curves.

\medskip

After briefly reviewing the general definition of the resonance variety, with a particular emphasis on the resonance constructed from vector bundles on curves, we recall the flatte\-ning stratification $\r(E) = \coprod_d \r_d(E)$ we will work with and its main properties, as presented in \cite{aprodu_linear_2025}. We then focus on rank-two bundles $E$ over elliptic curves, dividing our analysis into the split and non-split cases. The non-split case is straightforward to handle due to Atiyah’s classification, see Proposition \ref{prop:non_split}, whereas the split case requires a more detailed study. We begin by examining several specific bundles and computing the corresponding resonance in Proposition \ref{prop:particular_cases}, \ref{prop:2P_bQ} and Examples \ref{ex:2P_not_sim_2Q}, \ref{ex:2P_sim_2Q}. We then show in Theorem \ref{thm:general_split_bundles} --- together with Examples \ref{ex:3P_not_sim_3Q} and \ref{ex:3P_sim_3Q} --- that the resonance of a direct sum of two line bundles of degree at least $3$ coincides with the closure of the first non-empty stratum of its flattening stratification. Moreover, the resonance of such bundles has a unique irreducible component of maximal dimension. A similar analysis is carried out for the corresponding linear section $\g(E)$ of the Grassmannian $\gr(2, H^0(E))$, with special attention to its irreducible components, see Proposition \ref{prop:ired_comp_g(E)} and Corollary \ref{cor:ired_g(E)}. Throughout this note, we work over $\c$.

\section{Preliminaries}

\subsection{Resonance varieties} 

Let $V$ be a complex vector space of finite dimension and $K \subseteq \bigwedge^2V$ a linear subspace. 
Let us also consider $K^\perp := (\bigwedge^2V/K)^\vee$ the orthogonal complement of $K$. To these data one can attach the following uniruled projective variety
\[
\r(V,K) := \left\{[a]\in \p V^\vee : \exists \ b\in V^\vee \ \text{such that}\ 0\neq a \wedge b \in K^\perp \right\}
\]
which is called the (projectivized) \emph{resonance variety} of the pair $(V,K)$. 
This definition goes back to Papadima and Suciu \cite[\S 2]{papadima_vanishing_2015}, who introduced and worked with the affine cone over $\r(V,K)$. 
If $\g = \gr(2, V^\vee)$ denotes the Grassmann variety of all two-dimensional linear subspaces of $V^\vee$, seen as a projective variety via the Pl\"ucker embedding $\g \hooklongrightarrow \p(\bigwedge^2V^\vee)$, and 
\begin{equation} \label{eq:incidence}
\raisebox{-0.5\height}{
\begin{tikzpicture}
    \node (A) at (0,0) {$\{([a], \Lambda): a \in \Lambda\} \ =$};
    \node (B) [right = 0pt  of A] {$\Xi_V$};
    \node (G) [right = 20pt of B] {$\g$};
    \node (P) [below = 20pt of B] {$\p V^\vee$};
    
    \draw[->] (B) -- (G) node[midway, above] {$\scriptstyle \text{pr}_2$};
    \draw[->] (B) -- (P) node[midway, left] {$\scriptstyle \text{pr}_1$};
\end{tikzpicture}
}   
\end{equation}
is the incidence diagram of $\p V^\vee$ and $\g$, then $\r(V,K) = \text{pr}_1(\text{pr}_2^{-1}(\g \cap \p K^\perp))$.
In particular, the resonance is covered by lines.
As already noticed in \cite{papadima_vanishing_2015}, the resonance variety turns out to be the support locus of a finitely generated graded module $W(V,K)$ over the symmetric algebra of $V$, called the \emph{Koszul module} of the pair $(V,K)$. 
This algebro-geometric connection has been widely exploited, for instance, in \cite{aprodu_green_2019}, \cite{aprodu_reduced_2024}, \cite{aprodu_koszul_2024}. 
In this note, however, we will focus on the resonance itself.

\begin{rem} \label{rem:connected}
In general, if non-empty, the resonance $\r(V,K)$ is not necessarily an irreducible variety. However, if $\dim K < 2 \cdot \dim V - 4$, it is connected. Indeed, the last inequality is equivalent to
\[
\dim \g + \dim \p K^\perp > \dim \p(\textstyle{\bigwedge^2V^\vee}) \iff \codim \p K^\perp < \dim \g
\]
and by \cite[Theorem 2.1]{fulton_connectivity_1981} we obtain that the linear section $\g \cap \p K^\perp$ of the Grassmannian $\g$ is connected. Since $\Xi_V$ is a $\p^1$-bundle over $\g$, we immediately get the connectedness of the resonance by the incidence diagram (\ref{eq:incidence}).
\end{rem}

\subsection{Resonance of vector bundles over curves}

Let $X$ be a smooth projective variety and $E$ a vector bundle on $X$ of rank at least two. 
The pair $(X,E)$ gives rise to a resonance variety in the following natural way, see \cite{aprodu_koszul_2024}. Consider the second determinant map of $E$
\[
d_2: \textstyle{\bigwedge^2}H^0(X,E) \longrightarrow H^0(X,\textstyle{\bigwedge^2E})
\]
and take $V$ to be $H^0(X,E)^\vee$ and $K$ to be the orthogonal complement of the kernel of $d_2$. 
Put it differently, $V^\vee = H^0(X,E)$ and $K^\perp = \ker(d_2)$. 
We obtain in this way the resonance variety $\r(X,E): = \r(V,K)$ associated with vector bundle $E$. 
When the projective variety $X$ is clear from the context, we will simply write $\r(E)$ for the resonance of $E$.

By definition, $\r(E)$ is covered by the projective lines determined by the pairs of non-zero independent sections of $E$ generating a rank-one subsheaf of $E$. 
Indeed, the kernel of the determinant map is formed by those wedge products $s \wedge t$ such that $s$ and $t$ are linearly dependent in each fiber of $E$.

It is now an appropriate moment to fix some terminology.
By a \emph{sub-line bundle} of $E$ we mean a locally free subsheaf of $E$ of rank one. 
A \emph{subpencil} of $E$ is a sub-line bundle having at least two independent global sections.
A sub-line bundle $\cL$ is called \emph{saturated} if the quotient $E/\cL$ is torsion-free.
Note that any subsheaf $\cF \subseteq E$ can be saturated by considering 
\[
\cF^{\mathrm{sat}}:=\ker(E \longrightarrow (E/\cF)/(\mathrm{tors}(E/\cF))).
\]
$\cF^{\mathrm{sat}}$ is called the \emph{saturation} of $\cF$, see \cite[Definition 1.1.5]{huybrechts_2010}.

As explained in \cite[\S 6]{aprodu_reduced_2024} and \cite{aprodu_linear_2025}, the set $\w(E)$ of all saturated subpencils of $E$ plays an important role in the geometry of $\r(E)$. 
This set can be endowed with a natural projective scheme structure and one can define a map $\rho: \r(E) \longrightarrow \w(E)$ by sending the class of a resonant section $s \in H^0(X,E)$ to the saturation of the sheaf $s$ generates. 
The map $\rho$ may not be a morphism of algebraic varieties. 
However, a suitable stratification of $\w(E)$, which induces via pullback a stratification on the resonance, yields a finite family of natural morphisms between the corresponding strata. 
In the sequel, we shall describe this stratification, along with its main properties, in the case where $X$ has dimension one, following \cite[\S 4]{aprodu_linear_2025}.
This curve setup enjoys some important particular features. The general case was treated in \cite[\S 3]{aprodu_linear_2025}.

\medskip

Let now $C$ be a smooth projective curve and $E$ a vector bundle on $C$ of rank at least two. 
For any $d \in \z$, denote by $\w_d(E)$ the set of all saturated subpencils of $E$ of degree $d$ and by $\r_d(E)$ the inverse image of $\w_d(E)$ via $\rho$. 
The subsets $\r_d(E)$ form the \emph{flattening stratification} of the resonance
\[
\r(E) = \textstyle{\coprod_{d \ge 1}} \r_d(E)
\]
and the following result holds true, cf. \cite[Theorem 3.5]{aprodu_linear_2025}.

\begin{thm} \label{thm_flattenning_stratification}
The strata $\r_d(E)$ are locally closed and the maps induced by $\rho$
\[
\r_d(E) \xlongrightarrow{\rho} \w_d(E)
\]
are morphisms of algebraic varieties. Moreover, the strata $\w_d(E)$ and $\r_d(E)$ are non-empty for only finitely many values of $d$ and
\[
\overline{\r_d(E)} \subseteq \textstyle{\bigcup}_{e \ge d} \r_e(E).
\]
In particular, the stratum corresponding to the maximal value of $d$ such that $\r_d(E) \neq \emptyset$ is automatically closed.
\end{thm}

Furthermore, considering the closed subschemes
\[
W^1_d(E) := \{ L \in W^1_d(C) :  h^0(E(-L)) \neq 0 \}
\]
of the Brill-Noether loci $W^1_d(C) \subseteq \pic^d(C)$ consisting of all isomorphism classes of degree $d$ line bundles with at least two independent sections, we readily see, by the modularity property of the Jacobian of $C$, that there is a natural map
\[
\w_d(E) \longrightarrow W^1_d(E)
\]
sending a saturated subpencil $\cL$ of degree $d$ to its isomorphism class $L$ in the Picard group of $C$.
By Theorem \ref{thm_flattenning_stratification}, it follows that each stratum $\r_d(E)$ of the resonance is equipped with a morphism
\[
\rho_d : \r_d(E) \longrightarrow W^1_d(E).
\]
Note that the fiber of $\rho_d$ over a point $L \in W^1_d(E)$ is precisely the image of $\p(H^0(L)) \times U_L$ via the map
\[
\theta_L:\p(H^0(L))\times \p(H^0(E(-L))) \longrightarrow \p(H^0(E))
\]
induced by the multiplication of sections
\[
\mu_{L,E(-L)}:H^0(L)\otimes H^0(E(-L)) \longrightarrow H^0(E),
\]
where $U_L \subseteq \p(H^0(E(-L)))$ is the locus giving saturated embeddings of $L$ into $E$. 
Since being saturated is an open property, we infer the following.

\begin{rem} \label{rem:U_L}
If non-empty, $U_L$ is open dense in $\p(H^0(E(-L)))$ and the closure of the fiber of $\rho_d$ over a point $L \in W^1_d(E)$ is the image of $\theta_L$. Therefore, the closure of the stratum $\r_d(E)$ is set-theoretically the disjoint union of the images of $\theta_L$, for all $L \in W^1_d(E)$ such that $U_L$ is non-empty.
\end{rem}

\begin{rem} \label{rem_g(E)}
The linear section $\g(E): = \g \cap \p K^\perp$ of the Grassmannian $\g = \gr(2, H^0(E))$ admits a similar stratification by locally closed subsets as the resonance $\r(E)$, namely
\[
\g(E) = \textstyle{\coprod_{d \ge 1}} \g_d(E),
\]
where $\g_d$ is formed by those two-dimensional subspaces of $H^0(E)$ generating a subpencil of degree $d$. Each stratum $\g_d(E)$ admits a morphism to $W^1_d(E)$, whose fibers are the images of $\gr(2, H^0(L)) \times U_L$ via the map
\[
\tau_L:\gr(2, H^0(L)) \times \p(H^0(E(-L))) \longrightarrow \gr(2, H^0(E))
\]
induced by the multiplication map $\mu_{L,E(-L)}$. See \cite[\S 3-4]{aprodu_linear_2025} for details.
\end{rem}

\subsection{Rank-two vector bundles over an elliptic curve}

Let $C$ be an elliptic curve, that is, a smooth projective curve of genus $1$. 
Recall that $\pic^d(C) \cong  \pic^0(C) \cong C$ and the map $C \longrightarrow \pic^d(C)$, $P \longmapsto \o_C(dP)$ is a finite morphism of degree $d^2$, for any $d \in \z \setminus \{0\}$. 
Also, by Riemann-Roch, $h^0(L) = \deg L$, for any line bundle $L$ on $C$ of positive degree. 

Consider now a rank-two vector bundle $E$ over $C$. 
The bundle $E$ can be either split, that is $E = \o_C(aP) \oplus \o_C(bQ)$, where $a \le b$ are integers and $P, Q$ are closed points on $C$, or non-split, this latter case being classified by Atiyah \cite{atiyah_1957}, \cite[\S 5]{hartshorne_1977} as follows. 
If the degree of $E$ is even, $\deg E = 2d$, there exists a point $P \in C$ such that $E$ is the unique non-trivial extension
\begin{equation} \label{eq:even}
0 \longrightarrow \o_C(dP) \longrightarrow E \longrightarrow \o_C(dP) \longrightarrow 0
\end{equation}
and if the degree of $E$ is odd, $\deg E = 2d + 1$, there exists a pair of closed points $(P,Q)$ such that
\begin{equation} \label{eq:odd}
0 \longrightarrow \o_C(dP) \longrightarrow E \longrightarrow \o_C(dP + Q) \longrightarrow 0   
\end{equation}
Conversely, any integer $d$ and any closed point $P \in C$ determine a unique non-split rank-two bundle $E$ over the curve $C$. 
Furthermore, any extension by line bundles of $E$ has to be of type (\ref{eq:even}) or (\ref{eq:odd}).
In particular, if $\cL$ is a saturated sub-line bundle of a non-split bundle $E$, then the quotient $E/\cL$ is also a line bundle --- being a torsion-free sheaf on a curve --- and thus, $E$ can be written as an extension of line bundles
\[
0 \longrightarrow \cL \longrightarrow E \longrightarrow E/\cL \longrightarrow 0 .
\]
By Atiyah's classification, $\cL$ is isomorphic to $\o_C(dP)$, where $d = \lfloor \frac{\deg E}{2} \rfloor$.

\section{Description of the resonance}

In the sequel, we analyze the flattening stratification of the resonance of a rank-two bundle $E$ on an elliptic curve $C$.

We begin with the non-split case by considering $E$ a non-trivial extension of type (\ref{eq:even}) or (\ref{eq:odd}).
Since $\o_C(dP)$ is the unique saturated sub-line bundle of $E$, the resonance of $E$ is trivial if $d \le 1$.
Thus, we may assume that the degree of $E$ is at least $4$.
Note that $h^0(E(-dP)) = 1$, for otherwise the extension defining $E$ would be split.
Therefore, $\o_C(dP)$ embeds uniquely (up to scalar multiplication) in $E$ and the fiber over $\o_C(dP)$ of the map $\rho_d : \r_d(E) \longrightarrow W^1_d(E)$ is exactly $\p H^0(\o_C(dP)) \cong \p^{d-1}$.
Consequently, we obtain the following.

\begin{prop} \label{prop:non_split}
If $E$ is a non-split rank-two vector bundle over $C$,
\[
\r(E) = \begin{cases}
      \emptyset & \text{if}\ \deg E < 4; \\
      \p^{d-1}  & \text{if}\ \deg E \ge 4.
    \end{cases}  
\]
where $d = \lfloor \frac{\deg E}{2} \rfloor$.
\end{prop}

\begin{rem} \label{rem:ambient}
Note that the resonance of non-split rank-two vector bundles over an elliptic curve is a projective subspace. In particular, it is a linear variety. The ambient space of the resonance of such a vector bundle $E$ identifies with $\p^{\deg E - 1}$. Thus, if $E$ is of type (\ref{eq:even}), then $\r(E) \cong \p^{d - 1} \hooklongrightarrow \p^{2d - 1}$ and if $E$ is of type (\ref{eq:odd}), then $\r(E) \cong \p^{d - 1} \hooklongrightarrow \p^{2d}$.
\end{rem}

We now pass to the split case and consider $E = \o_C(aP) \oplus \o_C(bQ)$, with $a \le b$ and $P,Q \in C$. The following general result holds true.

\begin{lem} \label{lem:saturated_embedding}
If $L, L_1$ and $L_2$ are line bundles on a smooth projective curve $X$, then any saturated embedding of $L$ into the direct sum $L_1 \oplus L_2$ is determined by a pair of global sections $(s_1, s_2) \in H^0(L_1 \otimes L^\vee) \oplus H^0(L_2 \otimes L^\vee)$ without common zeros.
\end{lem}

\begin{proof}
The approach is standard, see for instance \cite[p.~32]{friedman_1998}. To begin with, note that giving an embedding of $L$ into $L_1 \oplus L_2$ is the same as giving a non-zero map
\[
\varphi : \o_X \xlongrightarrow{(s_1, s_2)} L' \oplus L'',
\]
where $L' = L_1 \otimes L^\vee$ and $L'' = L_2 \otimes L^\vee$. The embedding above is saturated if and only if the cokernel of
\[
\varphi_x : \o_{X,x} \xlongrightarrow{(s_1, s_2)} L'_x \oplus L''_x
\]
has no torsion for any $x\in X$. After choosing local trivializations, we may view $\varphi_x$ as a map
\[
\varphi_x : \o_{X,x} \xlongrightarrow{(s_1, s_2)} \o_{X,x} \oplus \o_{X,x}
\]
and derive that $\coker \varphi$ is torsion-free if and only if $s_1$ and $s_2$ are relatively prime in each stalk $\o_{X,x}$. Since $\o_{X,x}$ is a local ring of dimension one, this last condition is equivalent to saying that $s_1$ and $s_2$ do not share common zeros.
\end{proof}

This lemma immediately yields the following two results, which together provide a complete characterization of the saturated sub-line bundles of a split rank-two vector bundle on the elliptic curve $C$.

\begin{cor} \label{cor:saturated_embeddings_a<b}
If $E = \o_C(aP) \oplus \o_C(bQ)$ with $a<b$ and $L = \o_C(dR)$ is a degree $d$ line bundle on $C$, then
\begin{itemize}
    \item [(a)] If $d > b$, then $L$ cannot be embedded into $E$.
    \item [(b)] If $d = b$, then $L$ can be embedded into $E$ if and only if $bR \sim bQ$, that is $L = \o_C(bQ)$. In this case, $L$ embeds uniquely (up to scalar multiplication) and saturated in $E$.
    \item [(c)] If $a < d < b$, then $L$ cannot be embedded into $E$ as a saturated sub-line bundle.
    \item [(d)] If $d = a$, then $L$ can be embedded into $E$ as a saturated sub-line bundle if and only if $aR \sim aP$, that is $L = \o_C(aP)$. In this case, any saturated embedding of $L$ into $E$ is determined by a section of $H^0(\o_C(bQ - aP))$.
    \item [(e)] If $d < a$, then $L$ can be embedded into $E$ as a saturated sub-line bundle. Moreover, any saturated embedding of $L$ in $E$ is determined by a pair of two non-zero sections $(s_a, s_b) \in H^0(\o_C(aP - dR)) \oplus H^0(\o_C(bQ - dR))$ without common zeros. 
\end{itemize}
\end{cor}

\begin{cor} \label{cor_saturated_embeddings_a=b}
If $E = \o_C(aP) \oplus \o_C(aQ)$ and $L = \o_C(dR)$ is a degree $d$ line bundle on $C$, then
\begin{itemize}
    \item [(a)] If $d > a$, then $L$ cannot be embedded into $E$.
    \item [(b)] If $d = a$, then $L$ can be embedded into $E$ if and only if $aR \sim aP$ or $aR \sim aQ$, that is $L = \o_C(aP)$ or $L = \o_C(aQ)$. If $aP \not\sim aQ$, then such an $L$ embeds uniquely (up to scalar multiplication) and saturated in $E$. If $aP \sim aQ$, then the set of saturated sub-line bundles of $E$ isomorphic to $L = \o_C(aP)$ is parametrized by some $\p^1$.
    \item [(c)] If $d = a - 1$, then $L$ embeds into $E$ as a saturated sub-line bundle if and only if $aP \not\sim aQ$. In this case, any saturated embedding of $L$ into $E$ is determined by a pair of non-zero sections $(s_a, t_a) \in H^0(\o_C(aP - (a-1)R)) \oplus H^0(\o_C(aQ - (a-1)R))$.
    \item[(d)] If $d < a -1$, then $L$ can be embedded into $E$ as a saturated sub-line bundle. Moreover, any saturated embedding of $L$ into $E$ is determined by a pair of two non-zero sections $(s_a, t_a) \in H^0(\o_C(aP - dR)) \oplus H^0(\o_C(aQ - dR))$ without common zeros. 
\end{itemize}
\end{cor}

Since, for computing the resonance, we are interested in the saturated subpencils of $E$, that is, the saturated sub-line bundles of $E$ having at least two independent sections, we immediately obtain the following.

\begin{prop} \label{prop:particular_cases}
If $E = \o_C(aP) \oplus \o_C(bQ)$ and $a \le 1$, then
\[
\r(E) = \begin{cases}
      \emptyset & \text{if}\ b < 2; \\
      \p^{b-1}  & \text{if}\ b \ge 2.
    \end{cases}  
\]
\end{prop}

\begin{proof}
If $b < 2$, then $E$ has no subpencils and hence, the resonance of $E$ is trivial. If $b \ge 2$, by Corollary \ref{cor:saturated_embeddings_a<b}, it follows that the only saturated subpencil of $E$ is $\o_C(bQ)$ and 
\[
\r(E) = \r_b(E) = \p H^0 (\o_C(bQ)) \cong \p^{b-1}.
\]
\end{proof}

It remains to treat the cases where $a \ge 2$. Before addressing the general situation, let us examine a few examples.

\begin{exmp} \label{ex:2P_not_sim_2Q}
Let $E = \o_C(2P) \oplus \o_C(2Q)$, with $P, Q \in C$ and $2P \not\sim 2Q$. According to Corollary \ref{cor_saturated_embeddings_a=b}, the only saturated subpencils of $E$ are $\o_C(2P)$ and $\o_C(2Q)$ and
\[
\r(E) = \r_2(E) = \p H^0(\o_C(2P)) \sqcup \p H^0(\o_C(2Q)) \cong \p^1 \sqcup \p^1.
\]
Therefore, the resonance of $E$ is the disjoint union of $2$ lines in $\p H^0(E) \cong \p^3$. 
\end{exmp}

\begin{exmp} \label{ex:2P_sim_2Q}
Let $E = \o_C(2P) \oplus \o_C(2P)$, with $P \in C$. By Corollary \ref{cor_saturated_embeddings_a=b}, the line bundle $L = \o_C(2P)$ is the only saturated subpencil of $E$ and the fiber over $L$ of the morphism $\rho_2 : \r_2(E) \longrightarrow W^1_2(C)$ is the image of
\[
\theta_L : \p H^0(\o_C(2P)) \times \p H^0(\o_C \oplus \o_C) \longrightarrow \p H^0(E).
\]
Note that $\theta_L$ identifies with the Segre embedding $\sigma_{1,1}: \p^1 \times \p^1 \hooklongrightarrow \p^3$. Consequently, the resonance of $E$ may be viewed as the smooth quadric surface $\p^1 \times \p^1$ inside $\p H^0(E) \cong \p^3$.
\end{exmp}

\begin{rem} \label{rem:smooth_quadric}
It is interesting to note that the resonance from Example \ref{ex:2P_not_sim_2Q} coincides with the resonance of $\o_X(1,0) \oplus \o_X(0,1)$, where $X \cong \p^1 \times \p^1$ is a smooth quadric surface and the resonance from Example \ref{ex:2P_sim_2Q} coincides with the resonance of $\o_{\p^1}(1) \oplus \o_{\p^1}(1)$ on the projective line $\p^1$, see \cite[\S 7]{aprodu_linear_2025} for details.
\end{rem}

\begin{prop} \label{prop:2P_bQ}
Let $E = \o_C(2P) \oplus \o_C(bQ)$, with $b \ge 3$. Then $\overline{\r_2(E)}$ and $\r_b(E)$ are the irreducible components of the resonance of $E$.
\end{prop}

\begin{proof}
Due to Corollary \ref{cor:saturated_embeddings_a<b}, the only line bundles that can be embedded as saturated subpencils into $E$ are $L = \o_C(2P)$ and $M = \o_C(bQ)$. Therefore, the degree stratification of the resonance consists of two strata. The stratum $\r_2(E)$ is simply the fiber of $\rho_2$ over $L = \o_C(2P)$, which is the image of
\[
\theta_L : \p H^0(\o_C(2P)) \times U_L \longrightarrow \p H^0(E),
\]
where the locus $U_L \subseteq \p H^0(\o_C \oplus \o_C(bQ - 2P))$ of saturated embeddings of $L$ into $E$ may be identified with $H^0(\o_C(bQ - 2P))$. It can be readily seen that the map $\theta_L$ is injective and hence, the stratum $\r_2(E)$ is isomorphic to $\p^1 \times \a^{b-2} \hooklongrightarrow \p^{b+1}$.

The closed stratum $\r_b(E)$ is the fiber of $\rho_b$ over $M = \o_C(bQ)$, which is simply the image of the map
\[
\theta_M : \p H^0(\o_C(bQ)) \hooklongrightarrow \p H^0(E). 
\]
Therefore, $\r_b(E)$ identifies with $\p^{b-1} \hooklongrightarrow \p^{b+1}$. Since $\overline{\r_2(E)}$ and $\r_b(E)$ are irreducible of the same dimension, we infer that these are the irreducible components of $\r(E)$.
\end{proof}

We are left with the case where $a \ge 3$. We shall analyze the case $(a,b) = (3,3)$ separately, in the following two examples.

\begin{exmp} \label{ex:3P_not_sim_3Q}
Let $E = \o_C(3P) \oplus \o_C(3Q)$, with $P,Q \in C$ and $3P \not\sim 3Q$.  The resonance of $E$ consists of the strata $\r_2(E)$ and $\r_3(E)$. Let us begin with the closed stratum $\r_3(E)$. By Corollary \ref{cor_saturated_embeddings_a=b}, the only line bundles of degree $3$ that can be embedded in a saturated way into $E$ are $\o_C(3P)$ and $\o_C(3Q)$. Consequently, 
\[
\r_3(E) = \p H^0(\o_C(3P)) \sqcup \p H^0(\o_C(3Q)) \cong \p^2 \sqcup \p^2 \hooklongrightarrow \p^5.
\]
Let us pass to $\r_2(E)$ and first remark that any $L = \o_C(2R)$ can be embedded into $E$ as a saturated sub-line bundle and, in fact, $U_L \subseteq \p H^0(\o_C(3P - 2R) \oplus \o_C(3Q - 2R))$ is isomorphic to the projective line $\p^1$ with two points removed. Since the maps
\[
\theta_L : \p H^0(\o_C(2R)) \times \p H^0(\o_C(3P - 2R) \oplus \o_C(3Q - 2R)) \longrightarrow \p H^0(E)
\]
are quasi-finite, we infer that the morphism
\[
\rho_2 : \r_2(E) \longrightarrow W^1_2(C) = \pic^2(C) \cong C
\]
has equidimensional and irreducible fibers.
Now, since the curve $C$ is irreducible, there is an irreducible component $Z_0$ of $\r_2(E)$ that dominates $C$. Moreover, by the irreducibility of the fibers of $\rho_2$, we see that, in fact, $Z_0$ is the unique irreducible component of $\r_2(E)$ dominating $C$ and $\dim Z_0 = \dim \r_2(E) = 3$, by the fiber dimension theorem, see \cite[Theorem 1.25]{shafarevich_2013}.

Now, set-theoretically, the closure of $\r_2(E)$ is, by Remark \ref{rem:U_L}, the disjoint union of the images of the morphisms
\[
\theta_L : \p H^0(L) \times \p H^0(E(-L)) \longrightarrow \p H^0(E).
\]
We will show that $\r_3(E)$ is contained in the closure of the stratum $\r_2(E)$. 
Let us consider a non-zero section $s \in H^0(\o_C(3P))$ and write the divisor of zeros $D = (s)_0$ as $D_s + R'$, where $D_s$ is an effective divisor of degree $2$ and $R' \in C$. 
Let $t_s$ be the unique section (up to scalar multiplication) of $L_s = \o_C(3P - R')$ corresponding to $D_s$ and $t'$ be the unique section of $\o_C(R')$. 
Then
\[
\bigl( [t_s], [t' \oplus 0] \bigr) \xmapsto{\theta_{L_s}} [s \oplus 0].
\]
Consequently, $\p H^0(\o_C(3P)) \subseteq \overline{\r_2(E)}$. 
Similarly, $\p H^0(\o_C(3Q))$ is contained in $\overline{\r_2(E)}$. 
Furthermore, for dimensional reasons, these two copies of $\p^2$ are in fact contained in the closure of the maximal irreducible component $Z_0$ of the stratum $\r_2(E)$.
Thus, the resonance of $E$ is the closure of $\r_2(E)$.
\end{exmp}

\begin{exmp} \label{ex:3P_sim_3Q}
Let $E = \o_C(3P) \oplus \o_C(3P)$, with $P \in C$. By Corollary \ref{cor_saturated_embeddings_a=b}, the only line bundle with at least two independent global sections admitting a saturated embedding into $E$ is $L = \o_C(3P)$. Hence $\r(E) = \r_3(E)$, which is the image of 
\[
\theta_L : \p H^0(\o_C(3P)) \times \p H^0(\o_C \oplus \o_C) \longrightarrow \p H^0(E).
\]
Note that $\theta_L$ is just the Segre embedding $\sigma_{2,1}: \p^2 \times \p^1 \hooklongrightarrow \p^5$.
\end{exmp}

The arguments used in Example \ref{ex:3P_not_sim_3Q} can be extended to the remaining cases, as follows.

\begin{thm} \label{thm:general_split_bundles}
Let $E = \o_C(aP) \oplus \o_C(bQ)$, with $3 \le a \le b$, but $(a,b) \neq (3,3)$ and $P, Q \in C$. 
Then, the resonance of $E$ is equal to $\overline{\r_2(E)}$, which is a projective variety of dimension $a + b - 3$.
Moreover, $\r(E)$ has a unique irreducible component of maximal dimension.
\end{thm}

\begin{proof}
First of all, by Remark \ref{rem:U_L}, for any $d \ge 2$, the closure of the stratum $\r_d(E)$ is set-theoretically the disjoint union of the images of the maps
\[
\theta_L:\p H^0(L)\times \p H^0(E(-L)) \longrightarrow \p H^0(E)
\]
where $L \in W^1_d(E) = \pic^d(C)$ such that $U_L \neq \emptyset$, that is, $L$ admits a saturated embedding into $E$. The map $\theta_L$ is given by
\[
\bigl( [s], [s_a \oplus s_b] \bigr) \xmapsto{\theta_L} [s \cdot s_a \oplus s \cdot s_b].
\]
By Corollaries \ref{cor:saturated_embeddings_a<b} and \ref{cor_saturated_embeddings_a=b}, the locus $U_L$ is non-empty for any $L \in \pic^2(C)$.

Now, let us fix an integer $d \ge 2$, a line bundle $L \in W^1_d(E)$ such that $U_L \neq \emptyset$ and a non-zero section $s \in H^0(L)$. 
Writing the divisor of zeros $D = (s)_0$ as a sum of effective divisors $D_s + D'$, such that $\deg D_s = 2$ and $\deg D' = d - 2$ and taking the sections $t_s \in H^0(L(-D'))$ and $t' \in H^0(L(-D_s))$ such that $D_s = (t_s)_0$ and $D' = (t')_0$ we see that the sections $s$ and $t_s \cdot t'$ coincide up to scalar multiplication and
\[
\theta_L \bigl( [s], [s_a \oplus s_b] \bigr) = \theta_{L(-D')} \bigl( [t_s], [s_a \cdot t' \oplus s_b \cdot t'] \bigr).
\]
Therefore, $\overline{\r_d(E)}$ is contained in the closure of the stratum $\r_2(E)$. 
Furthermore, for any line bundle $L \in W^1_2(E)$, the map $\theta_L$ has finite fibers and the image of $\p H^0(L) \times U_L$ via $\theta_L$ is irreducible of dimension $a + b - 4$. 
Hence, the fibers of the morphism
\[
\rho_2: \r_2(E) \longrightarrow W^1_2(C) = \pic^2(C) \cong C
\]
onto the curve $C$ are irreducible of the same dimension. Thus, there is an irreducible component $Z_0$ of the stratum $\r_2(E)$ whose image through $\rho_2$ is dense in $C$. Furthermore, since the fibers of $\rho_2$ are irreducible, if a fiber of $\rho_2$ intersects $Z_0$, it is in fact entirely contained in $Z_0$ and hence, $Z_0$ is the unique irreducible component of $\r_2(E)$ dominating $C$. By \cite[Theorem 1.25]{shafarevich_2013}, the conclusion follows.
\end{proof}

\begin{rem}
It is worth noting that Theorem \ref{thm:general_split_bundles} strengthens the more general result proved in \cite[Proposition 3.7]{aprodu_linear_2025} in our particular setting: the irreducible components of $\r(E)$ are the irreducible components of the closure of the stratum $\r_2(E)$. Determining these components remains an open problem. We also note that, under the hypotheses of Theorem \ref{thm:general_split_bundles}, the resonance is connected. Indeed, since $\bigwedge^2 E \cong \o_C(aP +bQ)$ has precisely $a+b$ independent sections, the kernel of the determinant map of $E$ has dimension at least
\[
\dim \textstyle{\bigwedge^2}H^0(E) - \dim H^0(\textstyle{\bigwedge^2E}) =  \binom{a+b}{2} - (a+b).
\]
Therefore, the dimension of the linear subspace of $\bigwedge^2H^0(E)^\vee$ defining the resonance of $E$ is at most $a + b$. Now, since $a + b > 4$, the conclusion easily follows taking into consideration Remark \ref{rem:connected}.
\end{rem}

\begin{rem} \label{rem:non-empty_strata}
By Corollaries \ref{cor:saturated_embeddings_a<b} and \ref{cor_saturated_embeddings_a=b}, if $3\le a \le b$ and $E = \o_C(aP) \oplus \o_C(bQ)$, the non-empty strata of the resonance of $E$ --- corresponding to the non-empty strata $\w_d(E)$ --- are $\r_2(E), \ldots, \r_a(E)$ and $\r_b(E)$, except for the case where $a = b$ and $aP \sim aQ$. In the latter case, $\w_{a-1}(E) = \r_{a-1}(E) = \emptyset$.
\end{rem}

We now turn our attention to the corresponding linear section $\g(E)$ of the Grassmannian $\g = \gr(2, H^0(E))$ and analyze its stratification described in Remark \ref{rem_g(E)} in the case of a split vector bundle $E = \o_C(aP) \oplus \o_C(bQ)$, with $2 \le a \le b$. 

\begin{exmp} \label{ex:gr_2P_2Q}
If $E = \o_C(2P) \oplus \o_C(2Q)$, with $2P \not\sim 2Q$, then, by Corollary \ref{cor_saturated_embeddings_a=b} and Remark \ref{rem_g(E)}, $\g(E) = \g_2(E)$ consists of two points, whereas, if $E = \o_C(2P) \oplus \o_C(2P)$, then $\g(E) = \g_2(E)$ identifies with the projective line $\p^1$.
\end{exmp}

\begin{exmp} \label{ex:gr_2P_bQ}
If $E = \o_C(2P) \oplus \o_C(bQ)$, with $b \ge 3$, then, by Corollary \ref{cor:saturated_embeddings_a<b} and Remark \ref{rem_g(E)}, $\g(E)$ has two strata: $\g_2(E)$, which identifies with the affine space $\a^{b-2}$ and $\g_b(E)$, which identifies with the Grassmannian $\gr(2, b)$.
\end{exmp}

Now, let us consider $E = \o_C(aP) \oplus \o_C(bQ)$, such that $3 \le a < b$. By Remark \ref{rem_g(E)}, for each $d \in \{2, \ldots, a-1 \}$, there is a surjective morphism
\[
\g_d(E) \longrightarrow \pic^d(C) \cong C
\]
whose fibers are the images of the maps
\[
\tau_L : \gr(2, H^0(L)) \times U_L \longrightarrow \gr(2, H^0(E)),
\]
where $L \in \pic^d(C)$ and $U_L$ is the dense open subset of $\p H^0(E(-L))$ corresponding to the saturated embeddings of $L$ into $E$. Since the maps $\theta_L$ are quasi-finite, so are maps $\tau_L$, see \cite[Proposition 4.4]{aprodu_linear_2025}. Thus, arguing similarly as in Example \ref{ex:3P_not_sim_3Q} and Theorem \ref{thm:general_split_bundles}, the strata $\g_d(E)$ are of dimension 
\[
2(d - 2) + a + b - 2d = a + b - 4
\]
and have a unique irreducible component of maximal dimension. The other non-empty strata are $\g_a(E)$ --- which is isomorphic to $\gr(2,a) \times \a^{b-a}$ and has also dimension $a + b - 4$ --- and $\g_b(E)$, which can be identified with $\gr(2, b)$. Therefore, since the strata $\g_2(E), \ldots, \g_a(E)$ are equidimensional, the closures of their unique irreducible components of maximal dimension belong to the set of irreducible components of $\g(E)$, together with the closed stratum $\g_b(E)$.

It can be readily shown, by similar arguments, that if $E = \o_C(aP) \oplus \o_C(aQ)$, with $aP \not\sim aQ$, then the strata $\g_2(E), \ldots, \g_{a-1}(E)$ each have a unique irreducible component of maximal dimension $2a - 4$ and the closed stratum $\g_a(E)$ is simply a disjoint union of two copies of $\gr(2,a)$. An analogues statement holds for $E = \o_C(aP) \oplus \o_C(aP)$, except that $\g_{a-1}(E)$ is empty and the closed stratum is now isomorphic to $\gr(2, a) \times \p^1$.

Collecting the above, we obtain the following.

\begin{prop} \label{prop:ired_comp_g(E)}
If $E = \o_C(aP) \oplus \o_C(bQ)$, with $2 \le a \le b$, the number of irreducible components of $\g(E)$ is at least the number of non-empty strata $\g_d(E)$.
\end{prop}

\begin{cor} \label{cor:ired_g(E)}
If $E = \o_C(aP) \oplus \o_C(bQ)$, with $2 \le a \le b$, then $\g(E)$ is irreducible if and only if $a = b$, $a \in \{2, 3\}$ and $aP \sim aQ$.
\end{cor}

One cannot fail to notice the similarity between the results concerning $\r(E)$ and $\g(E)$ in the context of split rank-two bundles over elliptic curves and those for rank-two vector bundles over the projective line, see \cite[Propositions 4.8, 4.10]{aprodu_linear_2025}. Nevertheless, as shown in Proposition \ref{prop:2P_bQ}, the resonance of the vector bundles $\o_C(2P) \oplus \o_C(bQ)$, with $3 \le b$, is not irreducible, whereas the resonance of any rank-two bundle over $\p^1$ is always an irreducible variety.

We finally remark the extensive use of Corollaries \ref{cor_saturated_embeddings_a=b} and \ref{cor:saturated_embeddings_a<b} throughout our work. These are specific to the elliptic curves case. In higher genus, the problem is likely more delicate, and our present approach may not directly extend to a more general setting, even in the split case.

\printbibliography

\end{document}